\documentclass[12pt,a4paper]{amsart}
\usepackage[utf8]{inputenc}
\usepackage[T1]{fontenc}
\usepackage{amsmath,amsfonts,amssymb,amsthm,mathrsfs}
\usepackage{color}
\usepackage{marginnote}
\usepackage{constants}
\usepackage[hidelinks]{hyperref}
\newcommand{\be}{\begin{equation}} 
\newcommand{\ee}{\end{equation}}
\newcommand{\bea}{\begin{eqnarray}} 
\newcommand{\eea}{\end{eqnarray}}
\newcommand{\bean}{\begin{eqnarray*}} 
\newcommand{\eean}{\end{eqnarray*}}

\def\na{\nabla}

\def\mn{|\!\!|}

\def\mn2{|\!\!|_{M^{d/2}}}

\newtheorem{theorem}{Theorem}

\newtheorem{lemma}[theorem]{Lemma}

\theoremstyle{definition}

\theoremstyle{remark}
\newtheorem{remark}[theorem]{Remark}

\numberwithin{equation}{section}
\numberwithin{theorem}{section}

\newcommand{\normA}[2][]{\left\lVert #2 \right\rVert_{#1}}
\newconstantfamily{St1}{symbol=C}

\author[F. Heihoff]{Frederic Heihoff}
\address[F. Heihoff]{
Institut für Mathematik, Universität Paderborn, Warburger Str. 100, 33098 Paderborn, Germany\\
ORCID: 0000-0003-3654-0271
}
\email{fheihoff@math.uni-paderborn.de}

\author[P. Knosalla]{Piotr Knosalla*}
\address[P. Knosalla]{
Institute of Physics, University of Opole, Oleska 48, 45-052 Opole, Poland\\
ORCID: 0000-0002-3594-0938
}
\email{piotr.knosalla@uni.opole.pl}

\thanks{* Corresponding author: piotr.knosalla@uni.opole.pl}

\title[A chemotaxis-consumption model with boundary inflow]{A chemotaxis-consumption model with boundary inflow of a nutrient, steady states}

\begin{document}
\begin{abstract}
In this manuscript we consider the steady state problem for a chemotaxis-consumption model with positive inflow of a nutrient across the boundary. We show that in any bounded smooth domain there exists a unique nonconstant steady state provided that the mass of bacteria is sufficiently large. We further prove that, if the domain is a ball, this condition can be dropped and solutions exist for any positive bacterial mass.
\end{abstract}

\keywords{chemotaxis, consumption of chemoattractant, inflow of nutrient, nonconstant steady states}

\subjclass[2010]{35J25; 92C17; 35Q92}

\date{\today}
\maketitle

\baselineskip=16pt

\section{Introduction}

Let $\Omega\subset\mathbb{R}^n$ ($n\geq 1$) be a bounded domain with $C^{2,\alpha}$ boundary, $\alpha\in(0,1)$. We consider a colony of bacteria that lives in the container $\Omega$, which is otherwise filled with water containing a dissolved nutrient. The bacteria are assumed to walk randomly and follow the highest concentration of the nutrient, which diffuses and is consumed by said bacteria. To account for scenarios, in which the nutrient is abundant in the space surrounding our considered domain $\Omega$, e.g. in the common case of the nutrient being oxygen, we further presume a constant-in-time inflow of the nutrient across the boundary of the container $\partial\Omega$. By $u(x,t)$, $v(x,t)$, for a given point $x\in\bar{\Omega}$ and time $t\geq 0$, we denote the bacteria and nutrient density, respectively. Disregarding the effect of the fluid movement on the bacteria for the sake of simplicity, we arrive at the following reduced version of the model proposed by Tuval et al. \cite{Tuval} for settings of this kind,  

\begin{equation}\label{evol}
\left\{\begin{array}{lr}
u_t=\nabla\cdot(\nabla u-u\nabla v),&(x,t)\in\Omega\times(0,T),\\
\partial_\nu u-u\partial_\nu v=0,& (x,t)\in\partial\Omega\times(0,T),\\
v_t=\Delta v-uv,& (x,t)\in\Omega\times(0,T),\\
\partial_\nu v=\bar{v}>0,&(x,t)\in\partial\Omega\times(0,T),\\
u(x,0)=u_0(x)\geq 0,\,\,v(x,0)=v_0(x)\geq 0,& x\in\Omega,
\end{array}\right.
\end{equation}  
where $\bar{v}$ is a given positive constant representing the aforementioned inflow, and $\nu=\nu(x)$ is the outward normal vector to the boundary at point $x\in\partial \Omega$.

In this paper, we are interested in the existence and uniqueness of time-independent solutions of the system in \eqref{evol}, the so-called steady states or stationary solutions.   
If we denote by $(U(x),V(x)),$ $x\in\bar{\Omega},$ the steady-state solution of the above system, then the corresponding stationary problem for \eqref{evol} reads as follows: 

\begin{equation}\label{std}
\left\{\begin{array}{lr}
0=\nabla\cdot(\nabla U-U\nabla V),& x\in\Omega,\\
\partial_\nu U-U\partial_\nu V=0,& x\in\partial\Omega,\\
0=\Delta V-UV,& x\in\Omega,\\
\partial_\nu V=\bar{v}>0,& x\in\partial\Omega.\\
\end{array}\right.
\end{equation}  
A solution $(U,V)\in (C^2(\bar{\Omega}))^2$ of \eqref{std} is called positive if both its components $U,V$ are positive functions in $\bar\Omega$. 

Using the boundary condition for the first equation in \eqref{evol} we conclude that $\int_{\Omega}u(\cdot,t)=\int_\Omega u_0$ for all $t\in(0,T)$, i.e. solutions of this equation conserve the total mass of bacteria as one would expect. Therefore we consider \eqref{std} with an additional mass constraint on $U.$ Let thus $M\geq 0$ be a given and fixed total mass of bacteria, $M=\int_\Omega u_0=\int_\Omega U$. Notice that, if the pair $(U,V)\in (C^2(\bar{\Omega}))^2$ with $M>0$ is a solution of \eqref{std}, then both of its components are not constants as a direct consequence of their boundary conditions.

Testing the first equation in \eqref{std} by $Ue^{-V}$ and using the corresponding boundary condition, we get
\begin{align*}
0=\int_\Omega Ue^{-V}\nabla\cdot(e^V\nabla (Ue^{-V}))=&\int_{\partial\Omega}Ue^{-V}e^V\partial_\nu (Ue^{-V})-\int_\Omega e^{V}|\nabla(Ue^{-V})|^2\\
=&-\int_\Omega e^{V}|\nabla(Ue^{-V})|^2,
\end{align*}    
and thus we conclude that $\nabla (Ue^{-V})\equiv 0$, 
\begin{equation}\label{U:form}
U=Ce^V
\end{equation} 
for some constant $C$. 
 By integration of \eqref{U:form} we see that
$$
C=\frac{\int_\Omega U}{\int_\Omega e^V}=\frac{M}{\int_\Omega e^V},
$$   
and therefore obtain that the stationary bacteria density must have the form
\begin{equation}\label{U:gen:fo}
U=\frac{M}{\int_\Omega e^V}e^V.
\end{equation}
Using this we see that solving \eqref{std} becomes equivalent to solving the single equation, nonlocal problem
\begin{equation}\label{nonl:p}
\left\{\begin{array}{lr}
\Delta V=\frac{M}{\int_\Omega e^V}Ve^V,& x\in\Omega,\\
\partial_\nu V=\bar{v}>0,& x\in\partial\Omega.
\end{array}\right. 
\end{equation}
As a consequence of this analysis of \eqref{std} may be reduced to the study of \eqref{nonl:p}.  
The main result of this paper then reads as follows:

\begin{theorem}\label{MainR}
Suppose that $\Omega$ is a bounded domain with $\partial\Omega\in C^{2,\alpha}$. Then there exists a constant $L^\star\geq 0$ such that for $M>L^\star$ there exists a unique positive classical solution $(U,V)\in (C^{2,\alpha}(\bar{\Omega}))^2$ of \eqref{std}, which has the form
\begin{equation}\label{eqiv:rel}
(U,V)=\left(M\frac{e^V}{\int_\Omega e^V},V\right),
\end{equation}
and $V$ is a solution of \eqref{nonl:p}.	 If $\Omega=B_R(0)$, then $L^\star =0$ and the solution of \eqref{std} exists for any positive mass $M.$ This solution is radially symmetric, that is $(U(x),V(x))=(U(|x|),V(|x|))=(U(r),V(r))$ for $r\in[0,R)$, and the functions $U(r),$ $V(r)$ are both increasing and convex. 
\end{theorem}

\subsection*{Results related to the problem}

The system \eqref{evol} with $\bar{v}=0$ was analyzed by Tao and Winkler in \cite{TWin}. The authors proved the existence of global weak solutions that converge to a unique constant steady state $\left(\frac{M}{|\Omega|},0\right)$. Tao further showed in \cite{TConsumption} that classical solutions also exist if a smallness condition on the initial nutrient concentration is fulfilled. For an overview regarding further results for \eqref{evol} in this vein (including various variations on \eqref{evol}), we refer the reader to the survey found in \cite{LankeitWinklerD}. 
 
Braukhoff and Lankeit in \cite{BrLank} obtained the existence and uniqueness of a nonconstant solution for problem \eqref{std} under an inhomogeneous Robin boundary condition on $V$, and Fuest et al. in \cite{FLM} showed the global existence of a solution and its convergence to the steady state for the parabolic-elliptic counterpart of \eqref{evol}.

The local asymptotic stability of the solutions of \eqref{std} with a nonzero Dirichlet or Robin boundary condition was established by Li and Li in \cite{LiLi}.

Lee et al. in \cite{LeeWangYang} considered the existence, uniqueness and boundary layer profile of solutions of \eqref{std} with nonzero Dirichlet boundary condition on $V$. The stability of these steady states in one space dimension was studied by Hong and Wang in \cite{HongWang}.

The parabolic-elliptic counterpart of \eqref{evol} under Robin boundary condition for signal with a general tensor-valued sensitivity function was studied by Ahn, Kang and Lee in \cite{AhnKangLee} and global solvability was shown. The same system under a nonzero Dirichlet boundary condition with a general sensitivity function was considered by Yang and Ahn in \cite{YangAhn} and the global existence of a solution as well as its convergence to the steady state was established.

Ahn and Lankeit in \cite{AhnLankeit} considered the unique solvability of \eqref{std} with logarithmic sensitivity and a positive Dirichlet boundary condition imposed on $V$.

Recently Knosalla and Lankeit in \cite{KnLa} considered the parabolic-elliptic counterpart of the system \eqref{evol} with additional logistic growth of bacteria and a non-zero Dirichlet boundary condition for the signal. They studied the global existence of solutions and their convergence to the positive steady state obtained by Knosalla and Wr{\'o}bel in \cite{KW}. The problem under nonhomogenous Robin boundary condition was studied by Knosalla in \cite{Knos3} 

It is also worth pointing out the papers by Knosalla and Nadzieja, \cite{KNadz}, and by Knosalla, \cite{KNos,Knos1,Knos2} which concern the aerotaxis problem. The model considered there is a modified version of the system \eqref{evol} under a positive Dirichlet or Neumann boundary condition imposed on $V$. The results obtained in \cite{KNadz, Knos2} are closely related to \eqref{nonl:p}.

 Let us note that if $M=0,$ then problem \eqref{nonl:p} formally reduces to
\begin{equation}\label{V:nonex}
\left\{\begin{array}{lr}
\Delta V=0,& x\in\Omega,\\
\partial_\nu V=\bar{v}>0,& x\in\partial\Omega, 
\end{array}\right.
\end{equation}
which has no solution. This is the main difference to the Dirichlet or Robin boundary conditions already considered in the literature mentioned above, where the constant $\bar{v}$ is a solution corresponding to $M=0$. This fact causes serious obstructions in  obtaining existence of solutions to \eqref{std} and \eqref{nonl:p}.\\

\section{Proof of the main result} 
The proof of Theorem \ref{MainR} is split into several lemmas. As already established in the introduction we will mainly work with the nonlocal problem seen in \eqref{nonl:p} as it is easier to reason about, and then use the relation in \eqref{U:gen:fo} to gain \eqref{eqiv:rel} and thus a solution of (\ref{std}). To further break the construction of such a solution down into yet simpler parts, we note that solving \eqref{nonl:p} is equivalent to finding a solution $V_\lambda$ of the local semilinear problem
\begin{equation}\label{loc:p}
\left\{\begin{array}{lr}
\Delta V_\lambda=\lambda V_\lambda e^{V_\lambda},& x\in\Omega,\\
\partial_\nu V_\lambda=\bar{v}>0,& x\in\partial\Omega, 
\end{array}\right.
\end{equation}
with $\lambda\in(0,\infty)$, for which we have
\begin{equation}\label{nonl:cond}
M=\lambda\int_\Omega e^{V_\lambda}=:F(\lambda).
\end{equation}

Reframing our approach through \eqref{loc:p} splits it up into two fairly distinct parts. First we show that \eqref{loc:p} has a classical solution for all $\lambda \in (0,\infty)$. We then follow this up by an analysis of the function $F$ as it essentially determines the range of bacterial masses $M$ we can construct solutions for in this manner.

\begin{remark}\label{rem:compl}
Notably solutions to \eqref{loc:p} are not uniformly bounded in $L^\infty(\Omega)$ with respect to $\lambda>0$ as opposed to solutions to the corresponding Dirichlet or Robin problems. This is a natural consequence of the fact that \eqref{V:nonex} does not admit solutions and thus $\|V_\lambda\|_{L^\infty(\Omega)}$ must blow up when $\lambda$ tends to zero (see also Lemma \ref{lamlim}). This in turn means that it is a priori unclear what the value of $\lim_{\lambda\searrow 0}F(\lambda)$ is in the general case. This is the source of one of the major challenges in this manuscript given the importance of the characteristics of $F$, which we can only fully address in the radial case where $\lim_{\lambda\searrow 0}F(\lambda) = 0$ while in the general case we will only be able to show that $0 \leq \lim_{\lambda\searrow 0}F(\lambda) < \infty$.
\end{remark}
Before we get into our arguments proper, we need to make one final preparation in the form of the following straightforward global negativity 	result based on the Strong Maximum Principle and the Hopf Lemma \cite{GiT, HanLin}:  
\begin{lemma}\label{MP:lem}
	Suppose that $\Omega\subset\mathbb{R}^n$ is a bounded domain which has the interior sphere property and $w\in C^2(\Omega)\cap C^1(\bar{\Omega})$ is nonconstant and satisfies the differential inequalities
	\begin{equation}\label{MP:ineq}
		\Delta w+\na w\cdot b+ cw\geq 0\quad \text{ in }\Omega 
	\end{equation} 
	and 
	\begin{equation}\label{MP:dird}
		\partial_{\nu}w \leq 0 \quad \text{ in }\partial\Omega 
	\end{equation}  
	with $c\in L^\infty(\Omega),$ $b\in L^\infty(\Omega;\mathbb{R}^n),$ $c\leq 0$. Then
	\begin{equation}\label{MP:neg}
		w < 0 \quad \text{ in }\bar{\Omega}.
	\end{equation}
\end{lemma}
\begin{proof}
	Assume that there exists a point in $\bar{\Omega}$ at which $w$ is nonnegative. Then $w$ must attain a nonnegative maximum in $\bar\Omega$ due to continuity. But according to the Maximum Principle seen in \cite[Lemma 3.5]{GiT}, the fact that $w$ is nonconstant, $c \leq 0$ and \eqref{MP:ineq}, this nonnegative maximum cannot occur in the interior of $\Omega$. Therefore said nonnegative maximum must instead occur at a point $x_0$ on the boundary $\partial \Omega$ and must be strict in the sense that $u(x_0) > u(x)$ for all $x \in \Omega$. But this is exactly the kind of setting where the Hopf Lemma as laid out in \cite[Lemma 3.4]{GiT} applies and gives us that then $\partial_{\nu}w(x_0) > 0$ must hold true, which contradicts \eqref{MP:dird}. Therefore our initial assumption must have been wrong and we gain our desired result in \eqref{MP:neg}. 
\end{proof}

\subsection{Construction and properties of solutions $V_\lambda$ to (\ref{loc:p})}

We now present the promised existence and uniqueness proof for solutions of \eqref{loc:p}. 

\begin{lemma}\label{Exist:loc}
For $\lambda\in(0,\infty)$ there exists a unique positive solution $V_\lambda\in C^{2,\alpha}(\bar{\Omega})$ of \eqref{loc:p}.
\end{lemma} 
\begin{proof}
Our general strategy here will be to frame solutions to \eqref{loc:p} as fixed points of an appropriately chosen operator $\mathcal{F}$ on $C^0(\bar\Omega)$ and then apply Schaefer's Fixed Point Theorem, \cite[Theorem 9.2.2.4]{Evans}, which is a classic consequence of the well-known Leray-Schauder Theorem. To define the aforementioned operator, we introduce for each $f\in C^0(\bar{\Omega})$ the related linear problem 
\begin{equation}\label{loc:lin:aux}
\left\{\begin{array}{lr}
\Delta  W_{\lambda}-\lambda  W_{\lambda} = \lambda f( e^{f}-1),& x\in\Omega,\\
\partial_\nu  W_{\lambda}=\bar{v}>0,& x\in\partial\Omega.
\end{array}\right.
\end{equation}  
According to \cite[Theorem 2.4.2.7]{Grisvard}, there exists a unique solution $ W_{\lambda} \in \bigcap_{p\in(1,\infty)} W^{2,p}(\Omega) \subseteq C^0(\bar\Omega)$ of \eqref{loc:lin:aux} for each $f\in C^0(\bar\Omega)$ and we can thus define our desired operator $\mathcal{F} : C^0(\bar\Omega) \rightarrow C^0(\bar\Omega)$ by setting $\mathcal{F}[f] := W_{\lambda}$. It is then easy to see that any fixed points of this operator are solutions to (\ref{loc:p}).

As another consequence of \cite[Theorem 2.4.2.7]{Grisvard}, we can further for each $p\in(1,\infty)$ find $C = C(\Omega,p) > 0$ such that
\begin{equation*}
\begin{aligned}
\normA[W^{2,p}(\Omega)]{\mathcal{F}[f]} = \normA[W^{2,p}(\Omega)]{W_{\lambda}}&\leq  C\left(\normA[L^{p}(\Omega)]{\lambda f(e^f-1)}+\normA[W^{1-1/p,p}(\partial\Omega)]{\bar{v}}\right)\\
&\leq C\lambda(|\Omega|^\frac{1}{p} + |\partial\Omega|^\frac{1}{p})\left(\normA[L^{\infty}(\Omega)]{\lambda f(e^f-1)}+\bar{v}\right)\\
&\leq C\lambda(|\Omega|^\frac{1}{p} + |\partial\Omega|^\frac{1}{p})\left(\lambda e^{2\normA[L^{\infty}(\Omega)]{f}}+\bar{v}\right)
\end{aligned}  
\end{equation*}
for all $f \in C^0(\bar\Omega)$ and similarly that
\begin{equation*}
\begin{aligned}
\normA[W^{2,p}(\Omega)]{\mathcal{F}[f]-\mathcal{F}[g]}&\leq C\lambda\normA[L^{\infty}(\Omega)]{ f(e^{f}-1)-g(e^g-1)} \\
&\leq C\lambda e^{2\max\{\normA[L^{\infty}(\Omega)]{f}, \normA[L^{\infty}(\Omega)]{g}\}} \normA[L^{\infty}(\Omega)]{ f - g } 
\end{aligned}
\end{equation*}
for all $f,g \in C^0(\bar\Omega)$ by also employing the mean value theorem as well as a straightforward pointwise estimate. As for e.g.\ $p = n+1$ the space $W^{2,p}(\Omega)$ embeds continuously into $C^{1,\beta}(\bar\Omega)$ for some $\beta \in (0,1)$ due to the Sobolev inequality, \cite[Section 5.6.3]{Evans}, this is sufficient to ensure that $\mathcal{F}$ is both continuous and compact (by e.g.\ the Arzelà-Ascoli Theorem).

As the final prerequisite for Schaefer's Fixed Point Theorem, we now need to only show that the set 
$$
   \mathcal{M} := \left\{  V_{\lambda} \in C^0(\bar\Omega) \;|\; V_{\lambda} = \sigma\mathcal{F}[  V_{\lambda} ] \text{ for some } \sigma \in [0,1] \right\}
$$
is bounded in $C^0(\bar{\Omega})$. 

To this end, we will make use of the fact that all elements of $\mathcal{M}$ are  $W^{2,p}(\Omega)$-solutions to the related problem
\begin{equation}\label{loc:nlin:aux}
  \left\{\begin{array}{lr}
  \Delta V_{\lambda}-\lambda V_{\lambda}=\sigma\lambda  V_{\lambda} (e^{V_{\lambda}}-1),& x\in\Omega,\\
  \partial_\nu V_{\lambda}=\sigma\bar{v}>0,& x\in\partial\Omega
  \end{array}\right.
\end{equation}
for some $\sigma \in [0,1]$. Notice that for any $V_\lambda\in\mathcal{M}$ we have $V_\lambda \in \bigcap_{p\in(1,\infty)}W^{2,p}(\Omega) \subseteq \bigcap_{\beta \in (0,1)} C^{1,\beta}(\bar\Omega)$ due to the Sobolev inequality, \cite[Section 5.6.3]{Evans}. As this entails $\sigma\lambda V_{\lambda}(e^{V_{\lambda}}-1)\in C^{\alpha}(\bar\Omega)$, it thus follows by standard Schauder estimates, \cite{GNa}, that we have $V_{\lambda}\in C^{2,\alpha}(\bar{\Omega})$. It then immediately follows that $V_{\lambda}=0$ for $\sigma=0,$ and $V_{\lambda}$ is positive for all $\sigma \in (0,1]$ as a consequence of Lemma \ref{MP:lem}   since $-V_{\lambda}$ is a nonconstant solution of
$$
\Delta (- V_{\lambda}) + c (- V_{\lambda})=0\quad\text{ in }\Omega \quad\text{ and }\quad \partial_\nu (- V_{\lambda}) < 0 \text{ on }\partial\Omega
$$
with $c(x):=(\sigma -1)\lambda -\sigma\lambda e^{ V_{\lambda}(x)}<0$. To establish a corresponding upper bound, we let $\overline{W}_\lambda\in C^{2,\alpha}(\bar{\Omega})$ be the unique solution of  
\begin{equation}\label{supsol}
  \left\{\begin{array}{lr}
  \Delta \overline{W}_{\lambda}-\lambda \overline{W}_{\lambda}= 0,& x\in\Omega,\\
  \partial_\nu \overline{W}_{\lambda}=2\bar{v},& x\in\partial\Omega
  \end{array}\right.
\end{equation}
according to \cite[Theorem 2.4.2.7]{Grisvard} and \cite{GNa}.
Since $V_{\lambda}> 0$ in $\bar{\Omega}$ for $\sigma\in(0,1]$, the difference $V_{\lambda}-\overline{W}_{\lambda}$ satisfies both
$$
\Delta ( V_{\lambda}-\overline{W}_{\lambda}) -\lambda ( V_{\lambda}-\overline{W}_{\lambda})=\sigma\lambda  V_{\lambda}(e^{V_{\lambda}}-1)>0 \quad \text{ in }\Omega
$$
and 
$$
	\partial_\nu (V_{\lambda}-\overline{W}_{\lambda})|_{\partial\Omega} = (\sigma - 2)\overline{v} < 0\quad \text{ on } \partial \Omega.
$$
This lets us again apply Lemma \ref{MP:lem} to gain $V_{\lambda} < \overline{W}_\lambda$ in $\bar\Omega$. Thus it follows that $$
  \normA[L^\infty(\Omega)]{ V_{\lambda}} \leq \normA[L^\infty(\Omega)]{\overline{W}_{\lambda}}
$$ for all $V_{\lambda} \in \mathcal{M}$, which means that $\mathcal{M}$ is in fact bounded in $C^0(\bar{\Omega})$.

Therefore we can now apply \cite[Theorem 9.2.2.4]{Evans} to find $V_\lambda \in C^0(\bar\Omega)$ such that 
$$
  \mathcal{F}[ V_\lambda ] = V_\lambda,
$$
 and in fact $V_\lambda\in C^{2,\alpha}(\bar{\Omega}).$

Finally, we will now show that the solution found above must be unique. To do this, we assume that the problem in \eqref{loc:p} has two distinct solutions $V_{\lambda}$, $V_{\lambda}^*$. Then their difference $\tilde{V_\lambda}:= V_{\lambda}-V_{\lambda}^*$ satisfies 
\begin{equation}\label{dif:un:loc}
\left\{\begin{array}{lr}
\Delta \tilde{V_\lambda}-\lambda\tilde{V_\lambda}e^{V_{\lambda}}=\lambda V_{\lambda}^*(e^{V_{\lambda}}-e^{V_{\lambda}^*}),& x\in\Omega,\\
\partial_\nu \tilde{V_\lambda}=0,& x\in\partial\Omega. 
\end{array}\right.
\end{equation}   
Testing the first equation in \eqref{dif:un:loc} by $\tilde{V_\lambda}$ leads us to,
$$
-\int_\Omega|\nabla\tilde{V_\lambda} |^2-\lambda\int_\Omega \tilde{V_\lambda}^2 e^{V_{\lambda}}=\lambda\int_\Omega V_{\lambda}^*(e^{V_{\lambda}}-e^{V_{\lambda}^*}) \tilde{V_\lambda} = \lambda\int_\Omega V_{\lambda}^*(e^{V_{\lambda}}-e^{V_{\lambda}^*}) (V_{\lambda}-V_{\lambda}^*).
$$ 
Since the function $e^x$ is monotonically increasing and $V_{\lambda}^*>0$, we see that the right-hand side of the above equality is nonnegative. As a consequence we get that $\tilde{V_\lambda}\equiv 0,$ a contradiction.

\end{proof}

To further our understanding of the solutions $V_\lambda$ as well as better illustrate the point made in Remark \ref{rem:compl}, we now investigate their behavior in the two limit cases for the parameter $\lambda$.

\begin{lemma}\label{lamlim}
For the family of solutions $(V_\lambda)_{\lambda > 0}$ to \eqref{loc:p} constructed in Lemma \ref{Exist:loc}, we have 
\begin{equation}\label{pr1}
\lim_{\lambda\searrow 0}\normA[L^\infty(\Omega)]{V_\lambda}=+\infty,
\end{equation}
and
\begin{equation}\label{pr2}
\lim_{\lambda\rightarrow \infty}\normA[L^p(\Omega)]{V_\lambda}=0.
\end{equation}
for all $p \in [1,\infty]$ if $n = 1$ and $p \in [1,\infty)$ if $n \geq 2$.
\end{lemma}
\begin{proof} 
By integration of (\ref{loc:p}), we gain 
$$
   \bar{v}|\partial \Omega| = \lambda\int_\Omega V_{\lambda}e^{V_{\lambda}} \leq \lambda \normA[L^\infty(\Omega)]{V} e^{\normA[L^\infty(\Omega)]{V}}|\Omega|
$$
and thus 
$$
  \frac{1}{\lambda} \leq \frac{|\Omega|}{\bar v|\partial \Omega|} \normA[L^\infty(\Omega)]{V} e^{\normA[L^\infty(\Omega)]{V}}
$$
for all $\lambda > 0$. As the function $x \mapsto xe^x$ is monotonically increasing on $[0,\infty)$, this yields (\ref{pr1}).

Similarly integration of (\ref{loc:p}) also gives us 
$$
   \bar{v}|\partial \Omega| = \lambda\int_\Omega V_{\lambda}e^{V_{\lambda}} \geq \lambda \int_\Omega V_\lambda
$$ and thus 
\begin{equation}\label{asympt1}
  \int_\Omega V_\lambda\leq\frac{\bar{v}|\partial\Omega|}{\lambda}
\end{equation}
for all $\lambda > 0$. Using standard continuous trace embeddings, \cite[Theorem 5.5.1]{Evans}, combined with the Poincaré inequality, we can now fix $C_1 > 0$ such that 
\begin{equation}\label{eq:trace}
  \int_{\partial\Omega} \varphi^2 \leq C_1 \int_\Omega |\nabla \varphi|^2 + C_1\left(\int_\Omega \varphi \right)^2
\end{equation}
for all $\varphi \in W^{1,2}(\Omega)$. Testing the equation in (\ref{loc:p}) with $V^{p-1}_\lambda$, $p \in [2,\infty)$, we further gain
$$
  \bar{v}\int_{\partial \Omega} V_\lambda^{p-1} - \frac{4(p-1)}{p^2}\int_\Omega |\nabla V^\frac{p}{2}_\lambda|^2  = \lambda \int_\Omega V_\lambda^p e^{V_\lambda} \geq\lambda \int_\Omega V_\lambda^p
$$
for all $\lambda > 0$ by way of partial integration. Due to (\ref{eq:trace}) and Young's inequality, we can then see that  
$$
  \begin{aligned} 
  \bar{v}\int_{\partial \Omega} V^{p-1}_\lambda &\leq C_2(p) + \frac{2(p-1)}{p^2 C_1}\int_{\partial \Omega} V^p_\lambda  \\
  &\leq C_2(p) + \frac{2(p-1)}{p^2}\int_\Omega |\nabla V^\frac{p}{2}_\lambda|^2 + \frac{2(p-1)}{p^2} \left( \int_\Omega V^\frac{p}{2}_\lambda \right)^2
  \end{aligned}
$$
with $C_2(p) := \bar{v}^p|\partial\Omega|\left(\frac{p^2 C_1}{2(p-1)}\right)^{p-1}$, which combined with (\ref{asympt1}) lets us improve the above inequality to
\begin{equation}\label{asympt2}
\lambda \int_\Omega V^p_\lambda + \frac{2(p-1)}{p^2}\int_\Omega |\nabla V^\frac{p}{2}_\lambda|^2 \leq C_2(p) + \frac{2(p-1)}{p^2} \left( \int_\Omega V^\frac{p}{2}_\lambda \right)^2 
\end{equation}
for all $\lambda > 0$. Thus we gain the recursive inequality
\begin{equation*}
\normA[L^p(\Omega)]{V_\lambda} \leq \frac{C_3(p)}{\sqrt[p]{\lambda}} \left[ 1 + \normA[L^\frac{p}{2}(\Omega)]{V_\lambda} \right]
\end{equation*}
with $C_3(p) := \max\left(\sqrt[p]{C_2(p)}, \sqrt[p]{\frac{2(p-1)}{p^2}}\right)$ for all $p\in[2,\infty)$ and $\lambda > 0$. By now applying this recursive inequality iteratively along the sequence of exponents $p_n := 2^n$, $n \in \mathbb{N}_0$, and using the bound in (\ref{asympt1}) as the starting point, the property in (\ref{pr2}) immediately follows for all $p\in(1,\infty)$ and $n \geq 1$. As a further consequence of (\ref{asympt1}) and (\ref{asympt2}) with $p = 2$, we can find $C_4 > 0$ such that 
$$
  \normA[W^{1,2}(\Omega)]{V_\lambda} \leq C_4
$$
for all $\lambda \geq 1$. If $n = 1$, we can thus use the Gagliardo--Nireberg inequality to find $C_5 > 0$ such that  
$$
  \normA[L^\infty(\Omega)]{V_\lambda} \leq C_5 \normA[W^{1,2}(\Omega)]{V_\lambda}^\frac{2}{3}\normA[L^1(\Omega)]{V_\lambda}^\frac{1}{3}
$$
for all $\lambda \geq 1$, which readily implies (\ref{pr2}) for $p = \infty$  if $n = 1$ given that we have already shown that (\ref{pr2}) holds for $p =1 $ and thus completes the proof.
\end{proof}
\begin{remark}
  In one dimension, the property in \eqref{pr2} with $p = \infty$ can also be derived by using the fact that the solution $\overline{W}_\lambda$ to \eqref{supsol} from the proof of Lemma \ref{Exist:loc} is a supersolution for $V_\lambda$ and has an explicit form. In fact after normalizing the domain $\Omega$ to $(-R,R)$ by translation, which leaves the relevant differential equations unaffected, it is then easy to see that
  $$
    \overline{W}_\lambda(x) = 2\bar{v}\frac{e^{\sqrt{\lambda}x} + e^{-\sqrt{\lambda}x}}{\sqrt{\lambda}(e^{\sqrt{\lambda}R} - e^{-\sqrt{\lambda}R})} \quad \text{ for } x \in (-R, R)
  $$
  and thus $\lim_{\lambda\rightarrow \infty}\normA[L^\infty(\Omega)]{V_\lambda} \leq \lim_{\lambda\rightarrow \infty}\normA[L^\infty(\Omega)]{\overline{W}_\lambda} = \lim_{\lambda\rightarrow \infty}\overline{W}_\lambda(R) = 0$.
\end{remark}

\subsection{Properties of the function $F$}

Since we have at this point established the existence and uniqueness of solutions to \eqref{loc:p} for any $\lambda\in(0,\infty)$ as well as some of their key properties, we now shift our focus to the promised analysis of the function $F(\lambda) = \lambda \int_\Omega e^{V_\lambda}$ introduced in \eqref{nonl:cond}. As a first step in this endeavor, we show the differentiable dependence of $V_\lambda$ with respect to $\lambda$ with the help of the Implicit Function Theorem and obtain that its derivative $\frac{d}{d \lambda}V_\lambda$ is the solution of the linear elliptic problem in \eqref{V:lam:der}, which allows us to again make use of Lemma \ref{MP:lem} to gain rather useful positivity and negativity properties for some key terms involving $\frac{d}{d \lambda}V_\lambda$.

\begin{lemma}\label{V:dif:lem}
We have that $\frac{d}{d\lambda}V_{\lambda}\in C((0,\infty); C^{2,\alpha}(\bar{\Omega}))$ for the solutions $V_\lambda$ constructed in Lemma \ref{Exist:loc}. Further for any $\lambda\in(0,\infty)$ the function $\frac{d}{d\lambda}V_{\lambda}$ is a solution of the linear problem
\begin{equation}\label{V:lam:der}
\left\{\begin{array}{lr}
\Delta \left(\frac{d}{d\lambda}V_{\lambda}\right)-\lambda(1+V_\lambda)e^{V_\lambda}\left(\frac{d}{d\lambda}V_{\lambda}\right)=e^{V_\lambda}V_\lambda,& x\in\Omega,\\
\partial_\nu \left(\frac{d}{d\lambda}V_{\lambda}\right)=0,& x\in\partial\Omega. 
\end{array}\right.
\end{equation}
\end{lemma}

\begin{proof}
We begin by fixing a solution $\Phi \in C^{2,\alpha}(\bar{\Omega})$ of the problem
$$
\left\{\begin{array}{lr}
\Delta \Phi=\frac{|\partial\Omega|\bar{v}}{|\Omega|},& x\in\Omega,\\
\partial_\nu \Phi=\bar{v}>0,& x\in\partial\Omega, 
\end{array}\right.
$$
which is unique up to an additive constant, \cite{GNa}, and then use it to introduce the auxiliary problem 
\begin{equation}\label{nonl:loc:aux}
\left\{\begin{array}{lr}
\Delta H_\lambda= \lambda(H_\lambda+\Phi) e^{H_\lambda + \Phi}-\frac{|\partial\Omega|\bar{v}}{|\Omega|},& x\in\Omega,\\
\partial_\nu H_\lambda=0,& x\in\partial\Omega.
\end{array}\right.
\end{equation}
It is then easy to see that finding a solution to \eqref{loc:p} is equivalent to finding a solution to \eqref{nonl:loc:aux} using the relation $H_\lambda:=V_{\lambda}-\Phi$. Thus Lemma \ref{Exist:loc} ensures existence of unique solutions to the problem shown in (\ref{nonl:loc:aux}). We employ this transformation here because (\ref{nonl:loc:aux}) is more accessible to the following argument based on the Implicit Function Theorem to establish differentiable dependence of $H_\lambda$, and thus $V_\lambda$, on the parameter $\lambda$ due to its adjusted boundary condition.

We now consider the nonlinear operator $\mathcal{T}:(0,\infty)\times C^{2,\alpha}_\nu(\bar{\Omega})\rightarrow C^{\alpha}(\bar{\Omega})$ given by
$$
\mathcal{T}(\lambda, H):=\Delta H-\lambda (H+\Phi) e^{H+\Phi}+\frac{|\partial\Omega|\bar{v}}{|\Omega|},
$$
where $C_\nu^{2,\alpha}(\bar{\Omega})=\{f\in C^{2,\alpha}(\bar{\Omega}):\partial_\nu f|_{\partial\Omega}=0\}$. Using this operator, solutions to (\ref{nonl:loc:aux}) can be equivalently characterized as all $H_\lambda \in C_\nu^{2,\alpha}(\bar{\Omega})$ such that
\begin{equation}\label{op:eq:aux}
\mathcal{T}(\lambda, H_\lambda)=0
\end{equation}
holds. It is clear, by direct calculation, that $\mathcal{T}\in C^1((0,\infty)\times C^{2,\alpha}_\nu(\bar{\Omega});C^{\alpha}(\bar{\Omega}))$ and its partial Fr{\'e}chet derivatives are given by 
$$
D_H\mathcal{T}(\lambda, H)\varphi=\Delta\varphi-\lambda(1+H+\Phi)e^{H+\Phi}\varphi
$$
for $\varphi\in C^{2,\alpha}_\nu(\bar{\Omega})$ and
$$
D_\lambda\mathcal{T}(\lambda, H)=-(H+\Phi) e^{H+\Phi}.
$$
As a consequence of the same standard existence, \cite[Theorem 2.4.2.7]{Grisvard}, and Schauder theory, \cite{GNa}, as used in the proof of Lemma \ref{Exist:loc}, we can easily see that $D_H\mathcal{T}(\lambda, H)$ is invertible as long as $H + \Phi > -1$ on $\bar\Omega$. As $H_\lambda + \Phi = V_\lambda > 0$ by Lemma \ref{Exist:loc} for any solution $H_\lambda$ to (\ref{nonl:loc:aux}) with $\lambda>0$ and such a solution can be exactly characterized by (\ref{op:eq:aux}), this means, that for equation (\ref{op:eq:aux}) we can indeed apply the Implicit Function Theorem, \cite[Thm.\ 3.2.1]{AmAr}, at each point $(\lambda,H_\lambda)$ to gain our desired differentiable dependence of solution $H_\lambda$ on the parameter $\lambda$. Again due to the fact that $V_\lambda = H_\lambda + \Phi$, this then readily implies $\frac{d}{d\lambda}V_{\lambda}\in C((0,\infty); C^{2,\alpha}(\bar{\Omega}))$. It then follows that $\frac{d}{d\lambda}V_{\lambda}$ further solves (\ref{V:lam:der}) by straightforward direct computation or as another consequence of the Implicit Function Theorem.
\end{proof}

\begin{lemma}\label{V:lem}
For any $\lambda\in(0,\infty)$ we have
\begin{equation}\label{V:pr:1}
 \frac{d}{d\lambda}V_{\lambda}<0\quad \text{ in }  \bar{\Omega},
\end{equation}
and
\begin{equation}\label{V:pr:2}
1+\lambda \left(\frac{d}{d\lambda}V_{\lambda}\right)>0\quad \text{ in }  \bar{\Omega}
\end{equation} 
for the functions $V_\lambda$ constructed in Lemma \ref{Exist:loc}.
\end{lemma}
\begin{proof}  First, we show that $\frac{d}{d\lambda}V_{\lambda}$ cannot be a constant function in $\bar{\Omega}.$ Suppose, conversely, that $\frac{d}{d\lambda}V_{\lambda}$ is constant, then we immediately gain 
$$
\frac{d}{d\lambda}V_{\lambda} = -\frac{V_\lambda}{\lambda(V_\lambda + 1)}  \quad \text{ in } \bar{\Omega}
$$ 
as a consequence of \eqref{V:lam:der} and the fact that $V_\lambda > 0$. This is absurd, because, recalling that $V_\lambda$ is nonconstant, the right hand side of this equality is certainly a nonconstant function. 

 Now, from positivity of $V_\lambda$ and \eqref{V:lam:der} it follows that $\frac{d}{d\lambda}V_{\lambda}$ is a  nonconstant solution of the following differential inequality
$$
\left\{\begin{array}{lr}
\Delta \left(\frac{d}{d\lambda}V_{\lambda}\right)-\lambda(1+V_\lambda)e^{V_\lambda}\left(\frac{d}{d\lambda}V_{\lambda}\right)>0,& x\in\Omega,\\
\partial_\nu \left(\frac{d}{d\lambda}V_{\lambda}\right)=0,& x\in\partial\Omega, 
\end{array}\right.
$$
and thus by Lemma \ref{MP:lem} we get that $\frac{d}{d\lambda}V_{\lambda}<0$ in $\bar{\Omega}.$ Now we proceed to the proof of the second inequality. Notice that $-1-\lambda \left(\frac{d}{d\lambda}V_{\lambda}\right)$ satisfies 
$$
\left\{\begin{array}{lr}
\Delta\left(-1-\lambda \left(\frac{d}{d\lambda}V_{\lambda}\right)\right)-\lambda(1+V_\lambda)e^{V_\lambda}\left(-1-\lambda \left(\frac{d}{d\lambda}V_{\lambda}\right)\right)=\lambda e^{V_\lambda}>0,&x\in\Omega\\
\partial_\nu \left(-1-\lambda \left(\frac{d}{d\lambda}V_{\lambda}\right)\right)=0,& x\in\partial\Omega,
\end{array}\right.
$$
and thus again by Lemma \ref{MP:lem} we see that $-1-\lambda \left(\frac{d}{d\lambda}V_{\lambda}\right)<0$ in $\bar{\Omega}.$

\end{proof}

\begin{remark}
From \eqref{V:pr:1} we get the following ordering for solutions of \eqref{loc:p}:
\begin{equation}\label{pr3}
0<\lambda_1<\lambda_2\Rightarrow V_{\lambda_1}> V_{\lambda_2},\quad \text{ in }\bar{\Omega}.
\end{equation} 
\end{remark}

Making use of the monotony properties established in the above lemma, we can now derive some crucial properties of $F$ regarding its behavior in the general case. 

\begin{lemma}\label{limF}
The function $F(\lambda) := \lambda \int_\Omega e^{V_\lambda}$, $\lambda \in (0,\infty)$, introduced in \eqref{nonl:cond} is continuously differentiable, increasing,
\begin{equation}\label{F:pr1}
\lim_{\lambda\rightarrow\infty}F(\lambda)=+\infty,
\end{equation}
and there exists a nonnegative number $L^\star$ such that 
\begin{equation}\label{F:pr2}
\lim_{\lambda\searrow 0}F(\lambda)=L^\star.
\end{equation}

\end{lemma}
\begin{proof} 
The function $F$ is continuously differentiable since $V_\lambda\in C^1((0,\infty);C^{2,\alpha}(\bar{\Omega}))$ by Lemma \ref{V:dif:lem} and we also have 
$$F'(\lambda)=\int_\Omega \left(1+\lambda \frac{d}{d\lambda}
V_\lambda\right)e^{V_\lambda}\,\,>0,$$
by \eqref{V:pr:2}.
Since $V_\lambda>0$ we further gain
$$
F(\lambda)=\lambda\int_\Omega e^{V_\lambda}> \lambda |\Omega|,
$$
which implies \eqref{F:pr1}. Due to the fact that $F$ is increasing and $F(\lambda)\geq 0$ we get \eqref{F:pr2}.

\end{proof}

Now we address the special case $\Omega=B_R(0)$, where the above result can be significantly improved. Then, by uniqueness, $V_\lambda$ is radially symmetric, that is $V_\lambda(x)=V_\lambda(|x|)=:V_\lambda(r),$ and $V_\lambda(r)$ satisfies the boundary value problem
\begin{equation}\label{loc:p:rad}
\left\{\begin{array}{lr}
\frac{1}{r^{n-1}}(r^{n-1}V_\lambda'(r))'=\lambda V_\lambda(r) e^{V_\lambda(r)},& r\in(0,R),\\
V_\lambda '(0)=0,\,\, V_\lambda '(R)=\bar{v}>0,&  
\end{array}\right.
\end{equation}
where $'$ denotes the radial derivative $\frac{d}{d r}$. The boundary condition $V_\lambda '(0)=0$ in \eqref{loc:p:rad} follows from smoothness of $V_\lambda$ as well as symmetry.

\begin{lemma}
Any solution $V_\lambda$ of \eqref{loc:p:rad} is strictly increasing, convex, and
\begin{equation}\label{V:diff:est}
0<V_\lambda'(r)< \bar{v}
\end{equation}
for all $r\in(0,R).$
\end{lemma}

\begin{proof}
If we integrate once in \eqref{loc:p:rad} from $0$ to $r$ we obtain,
$$
V_\lambda'(r)=\frac{\lambda}{r^{n-1}}\int_0^rV_\lambda(s)e^{V_\lambda(s)}s^{n-1}ds,
$$
and thus $V_\lambda'(r)> 0$ for $r\in(0,R]$. Now,  substituting $t(s):=\frac{sR}{r}$,
\begin{align*}
V_\lambda'(r)
&=\frac{\lambda r}{R^n}
\int_0^R V_\lambda \left(\frac{tr}{R}\right) e^{V_\lambda(\frac{tr}{R})} t^{n-1}dt.
\end{align*}
Differentiating this expression with respect to $r$ we get,
\begin{align*}
V_\lambda''(r)=&
\frac{\lambda}{R^n}\int_0^R \left(V_\lambda \left(\frac{tr}{R}\right) +\frac{rt}{R} V_\lambda '\left(\frac{tr}{R}\right)\left(1+V_\lambda \left(\frac{tr}{R}\right)\right)\right) e^{V_\lambda(\frac{tr}{R})} t^{n-1}dt\\
&>0,
\end{align*}
for $r\in(0,R]$ and convexity of $V_\lambda$ follows. By convexity, we also find that $V_\lambda'(r)$ is increasing, and so \eqref{V:diff:est} follows. 

\end{proof}
Using the previous lemma and the mean value theorem, we can now significantly improve Lemma \ref{limF} in this case as promised.
\begin{lemma}\label{V:ball:lim}
If $\Omega=B_R(0)$, then
$$
\quad \lim_{\lambda \searrow 0} F(\lambda) = 0 \quad\text{ and thus }\quad L^\star = 0, 
$$
where $F$ is the function defined in \eqref{nonl:cond} and $L^\star$ is the number introduced in Lemma \ref{limF}.
\end{lemma}

\begin{proof}
Thanks to \eqref{V:diff:est} and the mean value theorem we have
\begin{equation}\label{V:inf:lb}
V_\lambda(R)-\bar{v}R\leq V_\lambda(0),
\end{equation}
and from \eqref{pr1} we know that 
\begin{equation}\label{V:rad:lim}
\lim_{\lambda\searrow 0}V_\lambda(R)=+\infty,
\end{equation}
 so the left hand side of \eqref{V:inf:lb} is evidently positive if $\lambda$ is small.
By integration of \eqref{loc:p} we get
$$
\bar{v}|\partial\Omega|=\lambda\int_\Omega V_\lambda e^{V_\lambda}\geq  F(\lambda) \inf_{\Omega} V_\lambda = F(\lambda) V_\lambda(0),
$$
and using \eqref{V:inf:lb} we obtain,
\begin{equation}\label{F:up:bd}
F(\lambda)\leq \frac{\bar{v}|\partial\Omega|}{V_\lambda(R)-\bar{v}R}
\end{equation}
for small $\lambda$. Thus, using \eqref{V:rad:lim} once again we see that $\lim_{\lambda \searrow 0} F(\lambda) = 0$ and therefore $L^\star=0$ in Lemma \ref{limF}.
\end{proof}

\subsection{Proof of Theorem \ref{MainR}}

Putting together the central results of the previous two sections, we can now give a fairly concise proof of our overall main result.

\begin{proof}[Proof of Theorem \ref{MainR}]
As we have shown in Lemma \ref{limF}, the function $F:(0,+\infty)\rightarrow(L^\star,+\infty)$ is well defined, continuous and strictly increasing, thus for each $M\in(L^\star,+\infty)$ the equation
$$
M=F(\lambda)
$$
has a unique solution $\lambda\in(0,+\infty)$, therefore the problem in \eqref{nonl:p} has a unique solution $V\in C^{2,\alpha}(\bar{\Omega})$ for any $M>L^\star$. We also see that actually \eqref{std} has a unique solution $(U,V)\in (C^{2,\alpha}(\bar{\Omega}))^2$, which has the form \eqref{eqiv:rel}. If we additionally assume that $\Omega=B_R(0),$ then the solution is necessarily radially symmetric and $L^\star=0$ by Lemma \ref{V:ball:lim}. Both components of the solution are increasing and convex. 
\end{proof}

\section*{Acknowledgments}

Piotr Knosalla was supported by the NCN grant Miniatura DEC-2024/08/X/ST1/01497. Frederic Heihoff was supported by Deutsche Forschungsgemeinschaft in the context of the project \emph{Fine structures in interpolation inequalities and application to parabolic problems}, project number 462888149.

\end{document}